\documentclass[11pt]{amsart}
\usepackage{amsmath}
\usepackage[active]{srcltx}
\usepackage{t1enc}
\usepackage[latin2]{inputenc}
\usepackage{verbatim}
\usepackage{amsmath,amsfonts,amssymb,amsthm}
\usepackage[mathcal]{eucal}
\usepackage{enumerate}
\usepackage[centertags]{amsmath}
\usepackage{graphics}
\usepackage{amssymb}
\usepackage{enumitem}

\makeatletter
\@namedef{subjclassname@2010}{  \textup{2010} Mathematics Subject Classification}
\makeatother
\newtheorem{theorem}{Theorem}[section]

\newtheorem{lemma}[theorem]{Lemma}

\theoremstyle{definition}
\newtheorem{remark}[theorem]{Remark}
\numberwithin{equation}{section}
\numberwithin{equation}{section}

\begin{document}
\title[On the divergence of subsequences of partial Walsh-Fourier sums]{On
the divergence of subsequences of partial Walsh-Fourier sums}
\author[U. Goginava and G. Oniani]{Ushangi Goginava and Giorgi Oniani}
\address{ I. Vekua Institute of Applied Mathematics and Faculty of Exact and
Natural Sciences of I. Javakhishvili Tbilisi State University, Tbilisi 0186,
2 University str., Georgia}
\email{zazagoginava@gmail.com}
\address{Department of Mathematics \\
Akaki Tsereteli State University \\
59 Tamar Mepe St., Kutaisi 4600\\
Georgia}
\email{oniani@atsu.edu.ge}
\date{}

\begin{abstract}
A class of increasing sequences of natural numbers $(n_k)$ is found for
which there exists a function $f\in L[0,1)$ such that the subsequence of
partial Walsh-Fourier sums $(S_{n_k}(f))$ diverge everywhere. A condition
for the growth order of a function $\varphi:[0,\infty)\rightarrow[0,\infty)$
is given fulfilment of which implies an existence of above type function $f$
in the class $\varphi(L)[0,1)$.
\end{abstract}

\subjclass[2010]{42C10}
\keywords{Walsh-Fourier series, subsequence of partial sums, everywhere
divergence}
\maketitle

\section{Definitions and Notation}

Let $r:\mathbb{R}\rightarrow \mathbb{R}$ be the function which is $1$%
-periodic and such that $r(x)=1$ if $x\in [0,1/2)$ and $r(x)=-1$ if $x\in
[1/2,1)$. The \emph{Rademacher system} $(r_n)_{n\in\mathbb{N}_0}$ is defined
as follows: $r_n(x)=r(2^nx)$ $(n\in\mathbb{N}_0, x\in[0,1))$. Here and below 
$\mathbb{N}_0$ denotes the set of all non-negative integers.

The \emph{binary coefficients} of a number $n\in\mathbb{N}_0$ will be
denoted by $\varepsilon_j(n)$ $(j\in\mathbb{N}_0)$, i.e. $%
\varepsilon_j(n)\in \{0,1\}$ for every $j\in\mathbb{N}_0$ and $%
n=\sum_{j=0}^\infty \varepsilon_j(n) 2^j$. Note that $\varepsilon_j(n)=0$ if 
$j$ is sufficiently large.

The \emph{Walsh system } $(w_n)_{n\in\mathbb{N}_0}$ is defined by 
\begin{equation*}
w_n(x)=\prod_{j=0}^\infty r_j(x)^{\varepsilon_j(n)}\;\;\;\;\;\;\;(n\in%
\mathbb{N}_0, x\in[0,1)).
\end{equation*}

The \emph{partial Walsh-Fourier sums} of a function $f\in L[0,1)$ are
defined as 
\begin{equation*}
S_{n}(f)=\sum_{k=0}^{n-1}\widehat{f}(k)w_{k}\;\;\;\;(n\in \mathbb{N}_{0}),
\end{equation*}%
where $\widehat{f}(k)=\int_{[0,1)}fw_{k}$.

The \emph{partial Fourier sums} (i.e. partial sums with respect to the
trigonometric system) of a function $f\in L[0,2\pi)$ will be denoted by $%
\mathbf{S}_n(f)$ $(n\in\mathbb{N})$.

Throughout the paper we will use the following convention: $\log n$ stands
for $\log_2 n$.

For $n\in\mathbb{N}_0$ let us denote 
\begin{equation*}
V(n)=\varepsilon_0(n)+ \sum_{j=1}^\infty
|\varepsilon_j(n)-\varepsilon_{j-1}(n)|.
\end{equation*}
The quantity $V(n)$ is called the \emph{variation of a number} $n$. It is
easy to check that 
\begin{equation*}
V(n)\leq C\log (2n)\;\;\;\;(n\in \mathbb{N})
\end{equation*}
and if $n_k=\sum_{j=0}^k 2^{2j}$ $(k\in\mathbb{N})$ then 
\begin{equation*}
c\log (2n_k)\leq V(n_k)\leq C \log (2n_k)\;\;\;\;(k\in \mathbb{N}).
\end{equation*}
Here $c$ and $C$ are absolute constants.

Note that for the Dirichlet kernels with respect to the Walsh system there
are valid the following estimations (see, e.g.. \cite{1}, p. 34) 
\begin{equation*}
V(n)/8\leq \Vert D_{n}\Vert _{L}\leq V(n)\;\;\;(n\in \mathbb{N}_{0}).
\end{equation*}

We will say that a sequence of natural numbers $(n_k)$ has \emph{bounded
variation} if $\sup_k V(n_k)<\infty$.

Let us define the \emph{spectrum of a number} $n\in \mathbb{N}_0$ as follows 
\begin{equation*}
\text{Sp}(n)=\{j\in\mathbb{N}_0:\varepsilon_j(n)=1\}.
\end{equation*}

Konyagin \cite{2} has considered increasing sequences of natural numbers $%
(n_{k})$ having the following property: 
\begin{equation*}
\max \text{Sp}(n_{k})<\min \text{Sp}(n_{k+1})\;\;\;\text{for every}%
\;\;\;k\in \mathbb{N}.
\end{equation*}%
We will refer such $(n_{k})$ as a \emph{sequence with separated spectrums}.

It is easy to see that $n_k=2^{k^2}\sum_{j=0}^k 2^{2j}$ $(k\in\mathbb{N})$
is a sequence with separated spectrums which has unbounded variation.

Let $(n_k)$ be an increasing sequences of natural numbers. We will call $%
(n_k)$ a \emph{sequence with nested spectrums } if 
\begin{equation*}
\text{Sp}(n_{k+1})\cap [0,\max \text{Sp} (n_{k})]=\text{Sp}(n_{k}) \;\;\;%
\text{for every}\;\;\;k\in\mathbb{N}.
\end{equation*}

It is easy to see that $n_k=\sum_{j=0}^k 2^{2j}$ $(k\in\mathbb{N})$ is a
sequence with nested spectrums which has unbounded variation.

Note that if $(n_k)$ is a sequence with nested spectrums which has unbounded
variation then its each subsequence $(m_k)$ also has the same properties.

Let $(n_k)$ be an increasing sequence of natural numbers with nested
spectrums which has unbounded variation. Denote by $\varphi_{(n_k)}$ the
function $\varphi_{(n_k)}:[0,\infty)\rightarrow [0,\infty)$ defined by the
following conditions:

\begin{enumerate}
\item[$\bullet$] $\varphi_{(n_k)}(2^{2n_\nu})=2^{2n_\nu} V(n_\nu)$ for every 
$\nu\in\mathbb{N}$;

\item[$\bullet$] $\varphi_{(n_k)}(0)=0$ and $\varphi_{(n_k)}$ is linear on
the segment $[0,2^{2n_1}]$;

\item[$\bullet$] $\varphi_{(n_k)}$ is linear on the segment $%
[2^{2n_\nu},2^{2n_{\nu+1}}]$ for every $\nu\in\mathbb{N}$.
\end{enumerate}

Obviously, $\varphi_{(n_k)}$ is a continuous and increasing function.
Furthermore, it is easy to check that 
\begin{equation*}
\varphi_{(n_k)}(u)\leq C u\log\log u\;\;\;\;(u\geq 8)
\end{equation*}
and if $n_k=\sum_{j=0}^k 2^{2j}$ $(k\in\mathbb{N})$ then 
\begin{equation*}
c u\log\log u\leq \varphi_{(n_k)}(u)\leq Cu\log\log u \;\;\;\;(u\geq 8).
\end{equation*}
Here $c$ and $C$ are absolute constants.

Let $(n_k)$ be an increasing sequence of natural numbers. We will say that a 
\emph{series} $\Sigma_j a_j$ \emph{converge} (\emph{diverge}) \emph{along} $%
(n_k)$ if the subsequence of partial sums $(S_{n_k})$ converge (diverge).

Let $\varphi:[0,\infty)\rightarrow[0,\infty)$ be a non-decreasing function
and $E\subset\mathbb{R}$ be a measurable set. By $\varphi(L)(E)$ it is
denoted the class of all measurable functions $f:E\rightarrow \mathbb{R}$
for which $\int_E \varphi(|f|)<\infty$.

By $\Phi$ we will denote the set of all non-decreasing functions $%
\varphi:[0,\infty)\rightarrow[0,\infty)$ for which $\liminf\limits_{u%
\rightarrow\infty}\varphi(u)/u>0$. Note that for every set $E\subset\mathbb{R%
}$ with finite measure and every function $\varphi\in\Phi$ it is valid the
inclusion $\varphi(L)(E)\subset L(E)$.

A sequence of positive numbers $(\alpha_k)$ is called \emph{lacunary} if
there is a number $\lambda>1$ such that $\alpha_{k+1}/\alpha_k\geq \lambda$
for every $k$.


\section{Results}

Kolmogoroff \cite{3} constructed a famous example of a function $f\in
L[0,2\pi )$ Fourier trigonometric series of which diverges almost
everywhere. After some years in the work \cite{4} he also constructed an
example of a function $f\in L[0,2\pi )$ for which the divergence takes place
everywhere.

Gosselin \cite{5} proved that for every increasing sequence of natural
numbers $(n_{k})$ there exists a function $f\in L[0,2\pi )$ such that 
\begin{equation}
\sup_{k\in \mathbb{N}}|\mathbf{S}_{n_{k}}(f)(x)|=\infty  \label{1}
\end{equation}%
for almost every $x\in \lbrack 0,2\pi )$. A function $f\in L[0,2\pi )$
satisfying (\ref{1}) for every $x\in \lbrack 0,2\pi )$ was constructed by
Totik \cite{6}.

Thus, there is no increasing sequence of natural numbers $(n_{k})$ along
which it is guaranteed almost everywhere convergence of Fourier
trigonometric series. The same is not true for Walsh-Fourier series. On the
one hand, an analogs of the examples of Kolmogoroff were constructed by
Stein \cite{7} and Schipp \cite{8,9}. On the other hand, if a sequence $%
(n_{k})$ has a bounded variation (for example, if $n_{k}=2^{k}$), then
Dirichlet kernels $D_{n_{k}}$ are uniformly bounded in the space $L[0,1)$.
It by using standard technique implies that the sums $S_{n_{k}}(f)$ converge
to $f$ almost everywhere for every function $f\in L[0,1)$. In this
connection Konyagin \cite{10} posed the problem:

\emph{Find a necessary and sufficient condition on an increasing sequence $%
(n_k)$ of natural numbers under which the partial Walsh-Fourier sums $%
(S_{n_k}(f))$ converge to $f$ almost everywhere for every function $f\in
L[0,1)$}.

Konyagin in the work \cite{2} established that the problem cannot be
resolved in terms of bounded variation of a sequence $(n_{k})$. Namely, in 
\cite{2} it was proved that if $(n_{k})$ is a sequence with separated
spectrums then $(S_{n_{k}}(f))$ converge to $f$ almost everywhere for every $%
f\in L[0,1)$.

The following theorem gives a sufficient condition implying the existence of
a Walsh-Fourier series divergent everywhere along a priori given sequence $%
(n_k)$.


\begin{theorem}
\label{t:1} If $(n_k)$ is a sequence of natural numbers with nested
spectrums which has unbounded variation then there exists a function $f\in
L[0,1)$ such that $\sup_{k\in \mathbb{N}} |S_{n_k}(f)(x)|=\infty \;\;\;\text{%
\emph{for every}}\;\;x\in[0,1).$
\end{theorem}


\begin{remark}
\label{r:1} Let us say that sequences $(m_{k})$ and $(n_{k})$ are \emph{close%
} if the sequence $(|m_{k}-n_{k}|)$ is bounded.

Let $f\in L[0,1)$ and $x\in [0,1)$. If for an increasing sequence of natural
numbers $(n_k)$ there is valid the relation $\sup_{k\in \mathbb{N}}
|S_{n_k}(f)(x)|=\infty$ then obviously the same relation is valid for every
sequence $(m_k)$ for which $(n_k)$ is a subsequence. Furthermore, it is easy
to check that the same relation is valid also for every sequence $(m_k)$
which is close to $(n_k)$.

Thus, the analog of Theorem \ref{t:1} is valid for every sequence $(n_{k})$
containing a subsequence with nested spectrums which has unbounded variation
or more generally, containing a subsequence close to some sequence with
nested spectrums which has unbounded variation.
\end{remark}


\begin{remark}
\label{r:2} From Theorem \ref{t:1} it follows the following corollary: If $%
(n_{k})$ is a sequence of natural numbers with nested spectrums which has
unbounded variation then for each its subsequence $(m_{k})$ there exists a
function $f\in L[0,1)$ such that $\sup_{k\in \mathbb{N}}|S_{m_{k}}(f)(x)|=%
\infty $ for every $x\in \lbrack 0,1)$.
\end{remark}


\begin{remark}
\label{r:3} From the corollary given in above remark it follows that
Walsh-Fourier series may diverge everywhere along an arbitrarily sparse
sequence, namely, for every sequence of positive numbers $(\lambda _{k})$
there exists a function $f\in L[0,1)$ and an increasing sequence of natural
numbers $(n_{k})$ such that $n_{k+1}/n_{k}\geq \lambda _{k}$ $(k\in \mathbb{N%
})$ and $\sup_{k\in \mathbb{N}}|S_{n_{k}}(f)(x)|=\infty $ for every $x\in
\lbrack 0,1)$. By Gát and Goginava \cite{11} it was proved the following
particular case of this result: For every sequence of positive numbers $(\nu
_{k})$ there exists a function $f\in L[0,1)$ and an increasing sequence of
natural numbers $(n_{k})$ such that $n_{k}\geq \nu _{k}$ $(k\in \mathbb{N})$
and $\sup_{k\in \mathbb{N}}|S_{n_{k}}(f)(x)|=\infty $ for every $x\in
\lbrack 0,1)$.
\end{remark}


\begin{remark}
\label{r:4} Gát \cite{12} proved that the arithmetic means of lacunary
partial Walsh-Fourier sums $\frac{1}{N}\sum_{k=1}^{N}S_{n_{k}}(f)$ (distinct
from lacunary partial Walsh-Fourier sums) are almost everywhere convergent
for every function $f\in L[0,1)$.
\end{remark}

Another topic we are interested is the problem on finding the optimal class $%
\varphi(L)$ in which partial Walsh-Fourier sums are almost everywhere
convergent along an a priori given sequence $(n_k)$.

The problems of above type has a rich history. For the "full" sequence $%
n_{k}=k$ fundamental importance have results of Kolmogoroff \cite{3,4} and
Carleson \cite{13} according to which almost everywhere convergence of
Fourier trigonometric series is not guaranteed in $L[0,2\pi )$ and is
guaranteed in $L^{2}[0,2\pi )$, respectively. They were improved and
extended to other orthonormal systems by various authors (see, e.g., \cite%
{10} or \cite{14} for the survey of the topic). The strongest results
concerning the trigonometric and the Walsh systems known nowadays are the
following:

\begin{enumerate}
\item[$\bullet $] If $\varphi \in \Phi $ and $\varphi (u)=o(u\sqrt{\log
u/\log \log u})$ $(u\rightarrow \infty )$ then there exists a function $f\in
\varphi (L)[0,2\pi )$ such that $\sup_{k\in \mathbb{N}}$ $|\mathbf{S}%
_{k}(f)(x)|=\infty $ for every $x\in \lbrack 0,2\pi )$ (Konyagin \cite{15});

\item[$\bullet $] If $f\in L\log ^{+}L\log ^{+}\log ^{+}\log ^{+}L[0,2\pi )$
then $\lim_{k\rightarrow \infty }\mathbf{S}_{k}(f)(x)=f(x)$ for almost every 
$x\in \lbrack 0,2\pi )$ (Antonov \cite{16});

\item[$\bullet $] If $\varphi \in \Phi $ and $\varphi (u)=o(u\sqrt{\log u})$ 
$(u\rightarrow \infty )$ then there exists a function $f\in \varphi (L)[0,1)$
such that $\sup_{k\in \mathbb{N}}|S_{k}(f)(x)|=\infty $ for every $x\in
\lbrack 0,1)$ (Bochkarev \cite{17});

\item[$\bullet $] If $f\in L\log ^{+}L\log ^{+}\log ^{+}\log ^{+}L[0,1)$
then $\lim_{k\rightarrow \infty }S_{k}(f)(x)=f(x)$ for almost every $x\in
\lbrack 0,1)$ (Sjölin, Soria \cite{18}).
\end{enumerate}

Regarding the general subsequences of partial Fourier sums Konyagin in \cite%
{19} proved the following theorem:

Let $(n_k)$ be an increasing sequence of natural numbers. Then for every $%
\varphi\in\Phi$ with $\varphi(u)=o(u\log\log u)$ $(u\rightarrow\infty)$
there exists a function $f\in \varphi(L)[0,2\pi)$ such that $\sup_{k\in 
\mathbb{N}} |\mathbf{S}_{n_k}(f)(x)|=\infty $ for every $x\in[0,2\pi)$.

By Konyagin in [10] it was conjectured that if $(n_k)$ is a lacunary
sequence then Fourier trigonometric series of every $f \in L\log^+\log^+ L
[0,2\pi)$ converge almost everywhere along $(n_k)$, i.e. it was conjectured
the optimality of the class $L\log^+\log^+ L [0,2\pi)$ for almost everywhere
convergence of Fourier trigonometric series along an arbitrary lacunary
sequence $(n_k)$.

To the study of the problem it was devoted several works. Suppose $(n_{k})$
is an arbitrary lacunary sequence. Do and Lacey \cite{20} have shown that if 
$f\in L\log ^{+}\log ^{+}L\log ^{+}\log ^{+}\log ^{+}L[0,1)$ then $%
\lim_{k\rightarrow \infty }S_{n_{k}}(f)(x)=f(x)$ for almost every $x\in
\lbrack 0,1)$. Lie \cite{21} proved an analogous theorem for the
subsequences of the partial Fourier sums. Di Plinio \cite{22} generalized
this results achieving almost everywhere convergence in the larger class 
\newline
$L\log ^{+}\log ^{+}L\log ^{+}\log ^{+}\log ^{+}\log ^{+}L$ both for the
Fourier and the Walsh-Fourier cases. Finally, Lie \cite{23} established that
if $\varphi \in \Phi $ with and \newline
$\varphi (u)=o(u\log \log u\log \log \log \log u)$ $(u\rightarrow \infty )$
then there exists a function $f\in \varphi (L)[0,2\pi )$ such that $%
\sup_{k\in \mathbb{N}}|\mathbf{S}_{n_{k}}(f)(x)|=\infty $ for almost every $%
x\in \lbrack 0,2\pi )$. Thus, the optimal class $\varphi (L)$ for the almost
everywhere convergence of Fourier trigonometric series along lacunary
sequences is $L\log ^{+}\log ^{+}L$ $\log ^{+}$ $\log ^{+}\log ^{+}$ $\log
^{+}L[0,2\pi )$ which is a bit smaller then the one conjectured in \cite{10}.

The problem of finding the optimal class $\varphi (L)$ in which it is
guaranteed the almost everywhere convergence of Walsh-Fourier series along
every lacunary sequence $(n_{k})$ is still open. Note that similar problem
may be posed for individual lacunary sequences as well. In this regard it is
true the following analogue of Konyagin's result from \cite{19} formulated
above.


\begin{theorem}
\label{t:2} Let $(n_{k})$ be an increasing sequence of natural numbers with
nested spectrums which has unbounded variation. Then for every $\varphi \in
\Phi $ with $\varphi (u)=o(\varphi _{(n_{k})}(u))$ $(u\rightarrow \infty )$
there exists a function $f\in \varphi (L)[0,1)$ such that $\sup_{k\in 
\mathbb{N}}|S_{n_{k}}(f)(x)|=\infty \;\;\;\text{for every}\;\;x\in \lbrack
0,1).$
\end{theorem}


\begin{remark}
\label{r:5} Let $n_{k}=\sum_{j=0}^{k}2^{2j}$ $(k\in \mathbb{N})$. Clearly,
the sequence $(n_{k})$ is lacunary. From Theorem \ref{t:2} it follows that
for every $\varphi \in \Phi $ with $\varphi (u)=o(u\log \log u)$ $%
(u\rightarrow \infty )$ there exists a function $f\in \varphi (L)[0,1)$ such
that $\sup_{k\in \mathbb{N}}|S_{n_{k}}(f)(x)|=\infty $ for every $x\in
\lbrack 0,1).$ Thus, in classes $\varphi (L)[0,1)$ with $\varphi \in \Phi $
and $\varphi (u)=o(u\log \log u)$ $(u\rightarrow \infty )$ it is not
guaranteed the almost everywhere convergence of Walsh-Fourier series along
lacunary sequences.
\end{remark}


\begin{remark}
\label{r:6} Obviously, Theorem \ref{t:2} implies Theorem \ref{t:1}.
\end{remark}

The proof of Theorem \ref{t:2} is based on modifications of methods proposed
by Schipp \cite{8,9} (see also \cite{1}, pp. 295-301) and Konyagin \cite{19}.


\section{Auxiliary propositions}

For every $n\in \mathbb{N}_0$ and $j=1,\dots,2^n$ by $\Delta(n,j)$ denote
the \emph{dyadic interval} $[(j-1)/2^n,j/2^n)$.

Polynomial with respect to the Walsh system we will refer simply as \emph{%
polynomial}.

The \emph{spectrum} of a polynomial $P=\sum_{j=0}^n a_j w_j$ will be denoted
by $\text{Sp}(P)$, i.e. $\text{Sp}(P)=\{j: a_j\neq 0\}$.

The \emph{degree} of a polynomial $P$ will be denoted by $\text{deg}(P)$.


\begin{lemma}
\label{l:1} Let $(n_k)$ be an increasing sequence of natural numbers with
nested spectrums which has unbounded variation. Then for every $\nu\in%
\mathbb{N}$ there exists a polynomial $P_\nu=1+\sum_{j=\nu}^{l_\nu}
a^{(\nu)}_j w_j$ and a set $E_\nu\subset [0,1)$ such that

\begin{enumerate}
\item[$\bullet$] $0\leq P_\nu(x)\leq 2^{2n_\nu}$ for every $x\in [0,1);$

\item[$\bullet $] $|E_{\nu }\cap \Delta (\max $ Sp$(n_{\nu })+1,j)|\geq
|\Delta (\max $ Sp$(n_{\nu })+1,j)|/4$ for every $j=1,\dots ,2^{\max \text{
Sp}(n_{\nu })+1};$

\item[$\bullet$] For every $x\in E_\nu$ there exists a natural number $%
k=k(\nu,x)\geq\nu$ such that $|S_{n_k}(P_\nu)(x)|\geq V(n_\nu)/16$.
\end{enumerate}
\end{lemma}

\begin{proof}

By $g$ denote the function defined as follows 
\begin{equation*}
g(x)=\text{sgn}(D_{n_{\nu }}(x))\;\;\;\;(x\in \lbrack 0,1)).
\end{equation*}%
Then we have 
\begin{equation*}
S_{n_{\nu }}(g)(0)=\Vert D_{n_{\nu }}\Vert _{L}\geq V(n_{\nu })/8.
\end{equation*}%
Denote $N=\max \text{Sp}(n_{\nu })+1$. It is clear that $g$ is a polynomial
with $\text{deg}(g)<2^{N}$. Consequently, for every point $x\in \Delta (N,1)$
we have also that 
\begin{equation*}
S_{n_{\nu }}(g)(x)=S_{n_{\nu }}(g)(0)\geq V(n_{\nu })/8.
\end{equation*}%
Denote $g_{j}(x)=g\big(x\oplus \frac{j-1}{2^{N}}\big)$ $(j=1,\dots
,2^{N};x\in \lbrack 0,1))$. Here and below $\oplus $ denotes the operation
of dyadic addition. Obviously, each $g_{j}$ is a polynomial with $\text{deg}%
(g_{j})=\text{deg}(g)<2^{N}<2n_{\nu }$. For every $j=1,\dots ,2^{N}$ and $%
x\in \Delta (N,j)$ we obtain 
\begin{equation}
|S_{n_{\nu }}(g_{j})(x)|=\Big|S_{n_{\nu }}(g)\Big(x\oplus \frac{j-1}{2^{N}}%
\Big)\Big|\geq V(n_{\nu })/8.  \label{2}
\end{equation}%
.

Let us consider a polynomial $Q$ of the following type 
\begin{equation*}
Q=\prod_{j=1}^{2^N}(1+w_{\delta_j} g_j),
\end{equation*}
where natural numbers $\delta_1<\delta_2<\dots<\delta_{2^N}$ will be chosen
later.

Obviously, 
\begin{equation*}
0\leq Q(x)\leq 2^{2^N}<2^{2n_\nu}\;\;\;\;(x\in [0,1)).
\end{equation*}

Denote $R_1=0$ and 
\begin{equation*}
R_j=w_{\delta_j}g_j\prod_{i=1}^{j-1}(1+w_{\delta_i} g_i)-w_{\delta_j}g_j
\;\;\;\;(j=2,\dots, 2^N).
\end{equation*}
Then it is valid the representation 
\begin{equation*}
Q=1+(w_{\delta_1}g_1+R_1) +\dots+ (w_{\delta_{2^N}}g_{2^N}+R_{2^N}).
\end{equation*}

Let $M$ be the minimal natural number with the properties: $M\geq N$ and $%
M\notin \bigcup_{k\in\mathbb{N}}\text{Sp}(n_k)$. Note that such number
exists since $(n_k)$ is a sequence with nested spectrums which has unbounded
variation.

Let us assume that numbers $\delta _{1},\dots ,\delta _{2^{N}}$ satisfy the
following conditions: 
\begin{equation}
\delta _{j}\in \{n_{k}-n_{\nu }:k>\nu \}\cup \{n_{k}-n_{\nu }+2^{M}:k>\nu
\}\;\;\;\;(j=1,\dots ,2^{N});  \label{3}
\end{equation}%
\begin{equation}
\delta _{1}\geq 2^{M}+1\;\;\text{and}\;\;\delta _{j+1}\geq 2(\delta
_{j}+2^{M})\;\;\;\;(j=2,\dots ,2^{N}).  \label{4}
\end{equation}

Below for the briefness we will use the notation $\lambda= 2^M-n_\nu$.

Let $1\leq j\leq 2^{N}$. Assume $g_{j}=\sum_{h=0}^{2^{N}-1}a_{h}w_{h}$. By
virtue of (\ref{4}) the numbers from $\text{Sp}(\delta _{j})$ are not less
then $N$. On the other hand for each $h<2^{N}$ the numbers from $\text{Sp}%
(h) $ are less then $N$. Consequently, 
\begin{equation*}
w_{\delta _{j}}w_{h}=w_{\delta _{j}\oplus h}=w_{\delta
_{j}+h}\;\;\;(h=0,\dots ,2^{N}-1)
\end{equation*}%
and 
\begin{equation*}
w_{\delta _{j}}g_{j}=\sum_{h=0}^{2^{N}-1}a_{h}w_{\delta
_{j}}w_{h}=\sum_{h=0}^{2^{N}-1}a_{h}w_{\delta _{j}+h}.
\end{equation*}%
Therefore, we have 
\begin{equation}
\text{Sp}(w_{\delta _{j}}g_{j})=\delta _{j}+\text{Sp}(g_{j})\subset \lbrack
\delta _{j},\delta _{j}+2^{N});  \label{5}
\end{equation}%
\begin{equation}
S_{\delta _{j}-\lambda }(w_{\delta _{j}}g_{j})=0  \label{6}
\end{equation}%
and 
\begin{equation}
S_{\delta _{j}+n_{\nu }}(w_{\delta _{j}}g_{j})=\sum_{h=0}^{n_{\nu
}-1}a_{h}w_{\delta _{j}+h}=w_{\delta _{j}}\sum_{h=0}^{n_{\nu
}-1}a_{h}w_{h}=w_{\delta _{j}}S_{n_{\nu }}(g_{j}).  \label{7}
\end{equation}

Let us prove that 
\begin{equation}
\text{Sp}(R_{j})\subset (\delta _{j-1}+2^{N},\delta
_{j}-2^{M})\;\;\;\;\;\;(j=2,\dots ,2^{N}).  \label{8}
\end{equation}%
Let us consider a product 
\begin{equation}
(w_{\delta _{j}}w_{h})(w_{\delta _{i(1)}}w_{h(1)})\dots (w_{\delta
_{i(t)}}w_{h(t)}),  \label{9}
\end{equation}%
where $j\geq 2$, $i(1)<\dots <i(t)<j$ and $h,h(1),\dots ,h(t)\leq \text{deg}%
(g)<2^{N}$. By virtue of (\ref{3}) and (\ref{4}) we have 
\begin{equation}
\delta _{j}\oplus \delta _{i(t)}\leq \delta _{j}-\delta _{i(t)}+2^{M}\leq
\delta _{j}-2\delta _{i(t-1)}-2^{M}.  \label{10}
\end{equation}%
From (\ref{4}) we obtain 
\begin{equation*}
\lbrack \delta _{i(1)}\oplus \dots \oplus \delta _{i(t-1)}]\oplus \lbrack
h\oplus h(1)\oplus \dots \oplus h(t)]<
\end{equation*}%
\begin{equation}
<2^{\max \text{Sp}(\delta _{i(t-1)})+1}\leq 2\delta _{i(t-1)},  \label{11}
\end{equation}%
and 
\begin{equation*}
\delta _{j}\oplus \lbrack \delta _{i(1)}\oplus \dots \oplus \delta
_{i(t)}]\oplus \lbrack h\oplus h(1)\oplus \dots \oplus h(t)]\geq
\end{equation*}%
\begin{equation}
\geq 2^{\max \text{Sp}(\delta _{j})}>\delta _{j}/2\geq \delta
_{j-1}+2^{M}\geq \delta _{j-1}+2^{N}.  \label{12}
\end{equation}%
Taking into account that $R_{j}$ is a linear combination of products of the
type (\ref{9}) and combining (\ref{10})-(\ref{12}) we obtain the inclusion (%
\ref{8}).

From (\ref{4}), (\ref{5}) and (\ref{8}) we conclude that the polynomial $Q$
is of the type $1+\sum_{j=\nu }^{l_{\nu }}a_{j}w_{j}$.

By virtue of (\ref{5})-(\ref{8}) for every $j$ we deduce 
\begin{equation*}
S_{\delta _{j}+n_{\nu }}(Q)-S_{\delta _{j}-\lambda }(Q)=w_{\delta
_{j}}S_{n_{\nu }}(g_{j}).
\end{equation*}%
Consequently, by virtue of (\ref{2}) for every $j$ and $x\in \Delta (N,j)$
we obtain 
\begin{equation*}
|S_{\delta _{j}+n_{\nu }}(Q)(x)-S_{\delta _{j}-\lambda }(Q)(x)|=|w_{\delta
_{j}}(x)S_{n_{\nu }}(g_{j})(x)|=
\end{equation*}%
\begin{equation}
=|S_{n_{\nu }}(g_{j})(x)|\geq V(n_{\nu })/8.  \label{13}
\end{equation}

(\ref{13}) implies that for each $j$ either 
\begin{equation}
|\{x\in \Delta (N,j):|S_{\delta _{j}+n_{\nu }}(Q)(x)|\geq V(n_{\nu
})/16\}|\geq |\Delta (N,j)|/2  \label{14}
\end{equation}%
or 
\begin{equation}
|\{x\in \Delta (N,j):|S_{\delta _{j}-\lambda }(Q)(x)|\geq V(n_{\nu
})/16\}|\geq |\Delta (N,j)|/2,  \label{15}
\end{equation}%
but till now we do not know exactly which between (\ref{14}) and (\ref{15})
is valid. We can make it clear by a new step of choosing of numbers $\delta
_{1},\dots ,\delta _{2^{N}}$. Namely, the new step for each $j$ will define
from which one between the sets $\{n_{k}-n_{\nu }:k>\nu \}$ and $%
\{n_{k}-n_{\nu }+2^{M}:k>\nu \}$ must be chosen a number $\delta _{j}$.

Let $\delta _{1}$ be an arbitrary number of the type $n_{k}-n_{\nu }$ where $%
k>\nu $ and such that $\delta _{1}\geq 2^{M}+1$. Suppose $\delta _{1},\dots
,\delta _{j-1}$ are already chosen. Taking into account $\,$(\ref{5}) and (%
\ref{8}) we write 
\begin{equation*}
S_{\delta _{j}-\lambda }(Q)=(1+w_{\delta _{1}}g_{1}+R_{1}+\dots +w_{\delta
_{j-1}}g_{j-1}+R_{j-1})+R_{j}=
\end{equation*}%
\begin{equation*}
=(1+w_{\delta _{1}}g_{1}+R_{1}+\dots +w_{\delta
_{j-1}}g_{j-1}+R_{j-1})+w_{\delta _{j}}R_{j}^{\ast },
\end{equation*}%
where $R_{j}^{\ast }=R_{j}/w_{\delta
_{j}}=g_{j}\prod_{i=1}^{j-1}(1+w_{\delta _{i}}g_{i})-g_{j}$. On the other
hand, using (\ref{5}), (\ref{7}) and (\ref{8}) we write 
\begin{equation*}
S_{\delta _{j}+n_{\nu }}(Q)=(1+w_{\delta _{1}}g_{1}+R_{1}+\dots +w_{\delta
_{j-1}}g_{j-1}+R_{j-1})+w_{\delta _{j}}R_{j}^{\ast }+w_{\delta
_{j}}S_{n_{\nu }}(g_{j}).
\end{equation*}%
Consequently, for every $x\in \Delta (N,j)$ we have 
\begin{equation*}
|S_{\delta _{j}+n_{\nu }}(Q)(x)-S_{\delta _{j}-\lambda
}(Q)(x)|=|(R_{j}^{\ast }+S_{n_{\nu }}(g_{j}))(x)-R_{j}^{\ast }(x)|\geq
V(n_{\nu })/8.
\end{equation*}%
It implies that 
\begin{equation}
|\{x\in \Delta (N,j):|(R_{j}^{\ast }+S_{n_{\nu }}(g_{j}))(x)|\geq V(n_{\nu
})/16\}|\geq |\Delta (N,j)|/2  \label{16}
\end{equation}%
or 
\begin{equation}
|\{x\in \Delta (N,j):|R_{j}^{\ast }(x)|\geq V(n_{\nu })/16\}|\geq |\Delta
(N,j)|/2.  \label{17}
\end{equation}

Note that from (\ref{4}) it follows easily the estimations 
\begin{equation*}
\deg (1+w_{\delta _{1}}g_{1}+R_{1}+\dots +w_{\delta _{j-1}}g_{j-1}+R_{j-1})<
\end{equation*}%
\begin{equation}
<2^{\max \text{Sp}(\delta _{j-1})+1}\leq 2^{\max \text{Sp}(\delta _{j})}.
\label{18}
\end{equation}

Suppose that (\ref{16}) is valid. From (\ref{18}) we conclude easily that
the numbers $(1+w_{\delta _{1}}g_{1}+R_{1}+\dots +w_{\delta
_{j-1}}g_{j-1}+R_{j-1})(x)$ and $w_{\delta _{j}}(x)(R_{j}^{\ast }+S_{n_{\nu
}}(g_{j}))(x)$ are of one and the same sign for points $x$ from a set $%
E_{\nu ,j}\subset \Delta (N,j)$ with $|E_{\nu ,j}|\geq |\Delta (N,j)|/4$.
Consequently, for every $x\in E_{\nu ,j}$ we have 
\begin{equation*}
|S_{\delta _{j}+n_{\nu }}(Q)(x)|\geq V(n_{\nu })/16.
\end{equation*}

The case, when (\ref{16}) is valid, is analogous to the considered one. In
this case we can guarantee that 
\begin{equation*}
|S_{\delta _{j}-\lambda }(Q)(x)|\geq V(n_{\nu })/16,
\end{equation*}%
for points $x$ from a set $E_{\nu ,j}\subset \Delta (N,j)$ with $|E_{\nu
,j}|\geq |\Delta (N,j)|/4$.

Now let us choose $\delta _{j}$ from the set $\{n_{k}-n_{\nu }:k>\nu \}$ if $%
(16)$ is fulfilled and from the set $\{n_{k}-n_{\nu }+2^{M}:k>\nu \}$ if (%
\ref{17}) is fulfilled. This finishes the process of choosing of numbers $%
\delta _{j}$.

Clearly, the polynomial $P_\nu=Q$ and the set $E_\nu=\bigcup_{j=1}^{2^N}E_{%
\nu,j}$ satisfy the needed conditions. The lemma is proved.
\end{proof}


\begin{remark}
\label{r:7} It is easy to see that the numbers $n_{k}$ from the third
condition of Lemma \ref{l:1} can be assumed to be not greater than $\min
\{n_{j}:j\in \mathbb{N},n_{j}\geq \text{deg}(P_{\nu })\}$.
\end{remark}

A function $\varphi:[0,\infty)\rightarrow [0,\infty)$ is said to satisfy $%
\Delta_2$-\emph{condition} if there are positive numbers $c$ and $u_0$ such
that $\varphi(2u)\leq c\varphi(u)$ when $u\geq u_0$.


\begin{lemma}
\label{l:2} Let $(n_{k})$ be an increasing sequence of natural numbers with
nested spectrums which has unbounded variation. Then the function $\varphi
_{(n_{k})}$ is convex, $\lim\limits_{u\rightarrow \infty }\varphi
_{(n_{k})}(u)/u=\infty $ and $\varphi _{(n_{k})}$ satisfies $\Delta _{2}$%
-condition.
\end{lemma}

\begin{proof}

Note that for a convex function $\varphi:[0,\infty)\rightarrow [0,\infty)$
the conditions $\lim\limits_{u\rightarrow \infty}\varphi(u)/u=\infty$ and $%
\lim\limits_{u\rightarrow \infty}D_r(\varphi)(u)=\infty$ are equivalent.
Here and below $D_r(\varphi)(x)$ denotes the right derivative of a function $%
\varphi$ at a point $x$.

Denote $I_0=[0, 2^{2n_1})$ and $I_\nu=[2^{2n_{\nu}},2^{2n_{\nu+1}})$ $(\nu\in%
\mathbb{N})$. By $t_\nu$ $(\nu\in\mathbb{N}_0)$ denote the slope of the
graph of the function $\varphi_{(n_k)}$ on the interval $I_\nu$, i.e., $%
t_0=V(n_1)$ and 
\begin{equation*}
t_\nu=\frac{2^{2n_{\nu+1}}V(n_{\nu+1})-2^{2n_{\nu}}V(n_{\nu})} {%
2^{2n_{\nu+1}}-2^{2n_{\nu}}}\;\;\;\;\;(\nu\in\mathbb{N}).
\end{equation*}

Let us prove that $(t_\nu)$ is an increasing sequence which tends to
infinity as $\nu\rightarrow\infty$. It will imply that $\lim\limits_{u%
\rightarrow \infty}D_r(\varphi_{(n_k)})(u)=\infty$ and that $\varphi_{(n_k)}$
is a convex function.

Denote $\delta_\nu=V(n_{\nu+1})-V(n_{\nu})$ $(\nu\in\mathbb{N})$. From the
properties of the sequence $(n_k)$ it follows that $\delta_\nu\geq 1$ $%
(\nu\in\mathbb{N})$.

For every $\nu \in \mathbb{N}$ we have 
\begin{equation}
t_{\nu }=V(n_{\nu })+\frac{2^{2n_{\nu +1}}\delta _{\nu }}{2^{2n_{\nu
+1}}-2^{2n_{\nu }}}.  \label{19}
\end{equation}%
Consequently, $t_{\nu }\rightarrow \infty $ as $\nu \rightarrow \infty $.
From the representation (\ref{19}) it follows directly that $t_{0}<t_{1}$.
Let us prove that $t_{\nu }<t_{\nu +1}$ for every $\nu \in \mathbb{N}$,
i.e., we must check the estimation 
\begin{equation}
V(n_{\nu })+\frac{2^{2n_{\nu +1}}\delta _{\nu }}{2^{2n_{\nu +1}}-2^{2n_{\nu
}}}<V(n_{\nu +1})+\frac{2^{2n_{\nu +2}}\delta _{\nu +1}}{2^{2n_{\nu
+2}}-2^{2n_{\nu +1}}}.  \label{20}
\end{equation}%
Let us rewrite (\ref{20}) as follows 
\begin{equation}
\bigg(\frac{2^{2n_{\nu +1}}}{2^{2n_{\nu +1}}-2^{2n_{\nu }}}-1\bigg)\delta
_{\nu }<\frac{2^{2n_{\nu +2}}\delta _{\nu +1}}{2^{2n_{\nu +2}}-2^{2n_{\nu
+1}}}  \label{21}
\end{equation}%
From the condition $\text{Sp}(n_{\nu +1})\cap \lbrack 0,\max \text{Sp}%
(n_{\nu })]=\text{Sp}(n_{\nu })$ it follows easily that $n_{\nu +1}\geq
n_{\nu }+2^{\max \text{Sp}(n_{\nu })+\delta _{\nu }}\geq n_{\nu }+\delta
_{\nu }$. Consequently, 
\begin{equation*}
\bigg(\frac{2^{2n_{\nu +1}}}{2^{2n_{\nu +1}}-2^{2n_{\nu }}}-1\bigg)\delta
_{\nu }\leq \frac{\delta _{\nu }}{2^{2\delta _{\nu }}-1}<1.
\end{equation*}%
Now taking into account that right part of (\ref{21}) is greater than 1, we
conclude the validity of the needed estimation.

Since $\varphi_{(n_k)}$ is linear on each interval $I_\nu$, then we easily
can check the validity of the estimation $\varphi_{(n_k)}(2^{m+1})\leq
2\varphi_{(n_k)}(2^{m})$ for every $m\in\mathbb{N}$. Hence, by monotonicity
of $\varphi_{(n_k)}$ we conclude that $\varphi_{(n_k)}$ satisfies $\Delta_2$%
-condition. The lemma is proved.
\end{proof}


\begin{lemma}
\label{l:3} Let $\alpha :[0,\infty )\rightarrow \lbrack 0,\infty )$ be an
increasing convex function with $\lim\limits_{u\rightarrow \infty
}\alpha(u)/u=\infty $ and let $\alpha $ satisfies $\Delta _{2}$-condition.
Then for every function $\beta :[0,\infty )\rightarrow \lbrack 0,\infty )$
with $\beta (u)=o(\alpha (u))$ $(u\rightarrow \infty )$ there exists a
function $\gamma :[0,\infty )\rightarrow \lbrack 0,\infty )$ such that: $1)$ 
$\gamma (0)=0$ and $\gamma $ is an increasing convex function; $2) $ $%
\lim\limits_{u\rightarrow \infty }\gamma(u)/u=\infty $; $3)$ $\gamma $
satisfies $\Delta _{2}$-condition; $4)$ there exists $u_{0}>0$ such that $%
\gamma (u)\geq \beta (u)$ when $u\geq u_{0}$; $5)$ $\gamma (u)=o(\alpha (u))$
$(u\rightarrow \infty )$.
\end{lemma}


\begin{remark}
\label{r:8} Lemma \ref{l:3} for the case $\alpha (u)=u\log ^{+}\log ^{+}u$
was proved in \cite{19} (see Lemma 3).
\end{remark}

\begin{proof}
Let us proceed the scheme used in \cite{19} (see the proof of Lemma 3 in 
\cite{19}).

Denote by $\alpha_j$ $(j\in\mathbb{N})$ the function $\alpha/2^j$. Taking
into account the properties of the functions $\alpha$ and $\beta$, it is
possible to choose an increasing sequence of positive numbers $(u_j)$ so
that: 
\begin{equation*}
\beta(u)\leq \alpha_j(u)\;\;\;\;(j\in\mathbb{N}, u\geq u_j);
\end{equation*}
\begin{equation*}
\alpha_j(u_j)+
D_r(\alpha_j)(u_j)(u_{j+1}-u_j)\leq\alpha_{j+1}(u_{j+1})\;\;\;\;(j\in\mathbb{%
N}).
\end{equation*}
Further for every $j\in\mathbb{N}$ we can choose numbers $v_j\in [u_j,
u_{j+1})$ and $M_j\in [D_r(\alpha_j)(v_j-0),D_r(\alpha_j)(v_j)] $ for which 
\begin{equation*}
\alpha_j(v_j)+ M_j(u_{j+1}-v_j)=\alpha_{j+1}(u_{j+1}).
\end{equation*}

Let us define a function $\gamma$ as follows: $\gamma(u)=2u\alpha_1(u_1)/u_1$
$(0\leq u< u_1)$; $\gamma(u)=2\alpha_j(u)$ $(j\in\mathbb{N}, u_j\leq u< v_j)$%
; and $\gamma(u)=2(\alpha_j(v_j)+ M_j(u-v_j))$ $(j\in\mathbb{N}, v_j\leq u<
u_{j+1})$. It is easy to see that $\gamma$ is an increasing convex function
and $\lim\limits_{u\rightarrow \infty}D_r(\gamma)(u)=\infty$. The rest three
properties of $\gamma$ follow from the estimations 
\begin{equation*}
\frac{1}{2^j}\leq \frac{\gamma(u)}{\alpha(u)} \leq\frac{1}{2^{j-1}}%
\;\;\;\;\;(j\in\mathbb{N}, u_j\leq u<u_{j+1}).
\end{equation*}
The lemma is proved.
\end{proof}

A function $\varphi:[0,\infty)\rightarrow [0,\infty)$ is called $N$-\emph{%
function} if $\varphi(0)=0$, $\varphi$ is increasing and convex, $%
\lim\limits_{u\rightarrow 0+} \varphi(u)/u=0$ and $\lim\limits_{u\rightarrow
\infty} \varphi(u)/u=\infty$.

Let us call functions $\alpha:[0,\infty)\rightarrow [0,\infty)$ and $%
\beta:[0,\infty)\rightarrow [0,\infty)$ \emph{equivalent} if there exist
positive numbers $u_0, c$ and $C$ such that $c \beta(u)\leq\alpha(u)\leq
C\beta(u)$ when $u\geq u_0$.


\begin{lemma}
\label{l:4} Let $\alpha :[0,\infty )\rightarrow \lbrack 0,\infty )$ be an
increasing convex function such that $\lim\limits_{u\rightarrow \infty
}\alpha (u)/u=\infty $. Then there exists an $N$-function $\beta $ which is
equivalent to $\alpha $.
\end{lemma}

\begin{proof}

For every $\varepsilon>0$ by $\alpha_\varepsilon$ denote the function on $%
[0,\infty)$ defined as follows: $\alpha_\varepsilon(u)=\varepsilon u^2$ when 
$0\leq u\leq 1$ and $\alpha_\varepsilon(u)=\alpha(u)- \alpha(1)+\varepsilon$
when $u>1$. It is easy to see that for any small enough $\varepsilon$ the
function $\beta=\alpha_\varepsilon$ satisfy the needed conditions. The lemma
is proved.
\end{proof}


\section{Proof of Theorem \protect\ref{t:2}}

By Lemmas \ref{l:2}, \ref{l:3} and \ref{l:4} we can assume that $\varphi $
is an $N$-function satisfying $\Delta _{2}$-condition.

Taking into account that $\lim\limits_{\nu \rightarrow \infty }V(n_{\nu
})=\infty $ and $\varphi (u)=o(\varphi _{(n_{k})}(u))$ $(u\rightarrow \infty
)$, we can choose the sequences $(M_{j})$ and $(n_{\nu (j)})$ with the
properties: 
\begin{equation}
M_{j}>0\;\;(j\in \mathbb{N})\;\;\;\text{and}\;\;\;\lim_{j\rightarrow \infty
}M_{j}=\infty ;  \label{22}
\end{equation}%
\begin{equation}
\sum_{j=1}^{\infty }\frac{M_{j}}{V(n_{\nu (j)})}\frac{\varphi (2^{n_{\nu
(j)}})}{2^{n_{\nu (j)}}}<\infty .  \label{23}
\end{equation}

Obviously, (\ref{23}) implies that $\sum_{j=1}^{\infty }M_{j}/V(n_{\nu
(j)})<\infty $. Let us consider a non-negative function $f^{\ast }$ defined
as follows: 
\begin{equation*}
f^{\ast }(x)=\sum_{j=1}^{\infty }\frac{M_{j}}{V(n_{\nu (j)})}P_{\nu
(j)}(x)\;\;\;\;(x\in \lbrack 0,1)),
\end{equation*}%
where $P_{\nu }$ $(\nu \in \mathbb{N})$ are polynomials from the Lemma \ref%
{l:1}. Let us show that $f^{\ast }\in \varphi (L)[0,1)$. Denote by $\psi $
the function conjugate to $\varphi $ in Young sense. Suppose, $g$ is an
arbitrary function from the class $\psi (L)[0,1)$. Using Young inequality we
have 
\begin{equation*}
\int_{\lbrack 0,1)}|f^{\ast }g|\leq \int_{\lbrack 0,1)}\sum_{j=1}^{\infty }%
\frac{M_{j}}{V(n_{\nu (j)})}|P_{\nu (j)}g|\leq
\end{equation*}%
\begin{equation*}
\leq \sum_{j=1}^{\infty }\frac{M_{j}}{V(n_{\nu (j)})}\int_{[0,1)}\varphi
(|P_{\nu (j)}|)\int_{[0,1)}\psi (|g|).
\end{equation*}%
Since $\varphi $ is a convex function then $\varphi (u)\leq (\varphi
(2^{2n_{\nu (j)}})/2^{n_{\nu (j)}})u$ when $0\leq u\leq 2^{n_{\nu (j)}}$. On
the other hand $0\leq P_{\nu }\leq 2^{2n_{\nu }}$ $(\nu \in \mathbb{N})$.
Consequently, for every $\nu \in \mathbb{N}$ we obtain 
\begin{equation*}
\int_{\lbrack 0,1)}\varphi (|P_{\nu (j)}|)\leq \int_{\lbrack 0,1)}\frac{%
\varphi (2^{2n_{\nu (j)}})}{2^{n_{\nu (j)}}}P_{\nu (j)}=\frac{\varphi
(2^{2n_{\nu (j)}})}{2^{n_{\nu (j)}}}.
\end{equation*}%
Therefore by (\ref{23}) we write 
\begin{equation}
\int_{\lbrack 0,1)}|f^{\ast }g|\leq \sum_{j=1}^{\infty }\frac{M_{j}}{%
V(n_{\nu (j)})}\frac{\varphi (2^{2n_{\nu (j)}})}{2^{n_{\nu (j)}}}%
\int_{[0,1)}\psi (|g|)<\infty .  \label{24}
\end{equation}%
Taking into account that $g$ is an arbitrary function of the class $\psi
(L)[0,1)$ from (\ref{24}), we conclude that $f^{\ast }$ belongs to the
Orlicz space $L^{\varphi }[0,1)$ generated by the function $\varphi $ (see,
e.g., \cite{23}, Ch. II). Since $\varphi $ satisfies $\Delta _{2}$-condition
then the class $\varphi (L)[0,1)$ coincides with $L^{\varphi }[0,1)$. Thus,
we proved that $f^{\ast }\in \varphi (L)[0,1)$.

Denote $N_{\nu }=\min \{n_{j}:j\in \mathbb{N},n_{j}\geq \text{deg}(P_{\nu
})\}$ $(\nu \in \mathbb{N})$. Let us assume that the sequences $(M_{j})$ and 
$(n_{\nu (j)})$ together with the conditions (\ref{22}) and (\ref{23}) also
satisfy the following one: For every $j\in \mathbb{N}$ the number $\nu (j+1)$
is greater than $N_{\nu (j)}$ and 
\begin{equation}
\text{card}\Big((N_{\nu (j)},\nu (j+1))\cap \{n_{k}:k\in \mathbb{N}\}\Big)%
\geq 2.  \label{25}
\end{equation}%
For every $j\in \mathbb{N}$ let us choose two natural numbers $\alpha
(j)<\beta (j)$ for which $n_{\alpha (j)}$ and $n_{\beta (j)}$ belong to the
intersection given in (\ref{25}).

Let us prove that 
\begin{equation}
\sup_{j\in \mathbb{N}}\;\sup_{k:\nu (j+1)\leq n_{k}\leq N_{\nu
(j+1)}}|S_{n_{k}}(f^{\ast })(x)-S_{n_{\beta (j)}}(f^{\ast })(x)|=\infty
\label{26}
\end{equation}%
for almost every $x\in \lbrack 0,1)$, and 
\begin{equation}
S_{n_{\beta (j)}}(f^{\ast })(x)-S_{n_{\alpha (j)}}(f^{\ast })(x)=0
\label{27}
\end{equation}%
for every $j\in \mathbb{N}$ and $x\in \lbrack 0,1)$.

Let $j\in \mathbb{N}$, $x\in E_{\nu (j+1)}$ and $k=k(\nu (j+1),x)$, where $%
E_{\nu (j+1)}$ and $k(\nu (j+1),x)$ are parameters from Lemma \ref{l:1}.
Note that for every $n,m\in \mathbb{N}$ and $x\in \lbrack 0,1)$, 
\begin{equation*}
S_{n}(f^{\ast })(x)-S_{m}(f^{\ast })(x)=
\end{equation*}%
\begin{equation}
=\sum_{i=1}^{\infty }\frac{M_{i}}{V(n_{\nu (i)})}(S_{n}(P_{\nu
(i)})(x)-S_{m}(P_{\nu (i)})(x)).  \label{28}
\end{equation}%
Hence, by virtue of (\ref{25}) and Lemma \ref{l:1} we obtain 
\begin{equation*}
|S_{n_{k}}(f^{\ast })(x)-S_{n_{\beta (j)}}(f^{\ast })(x)|=
\end{equation*}%
\begin{equation*}
=\frac{M_{j+1}}{V(n_{\nu (j+1)})}|S_{n_{k}}(P_{\nu (j+1)})(x)-S_{n_{\beta
(j)}}(P_{\nu (j+1)})(x)|\geq
\end{equation*}%
\begin{equation*}
\geq \frac{M_{j+1}}{V(n_{\nu (j+1)})}\left( \frac{V(n_{\nu (j+1)})}{16}%
-1\right) .
\end{equation*}%
Consequently, taking into account (\ref{22}) and (\ref{23}), we conclude
that (\ref{26}) holds for every point $x$ which belongs to infinite number
of sets $E_{\nu (j+1)}$ $(j\in \mathbb{N})$. Using the uniform lower
estimation for measures of portions of sets $E_{\nu }$ in dyadic intervals
(see the second estimation in Lemma \ref{l:1}) it is easy to see that $%
\bigcup_{j=N}^{\infty }E_{\nu (j+1)}$ is a set of full measure for every $%
N\in \mathbb{N}$. Therefore, the upper limit of the sequence $(E_{\nu
(j+1)}) $ is a set of full measure. Thus, (\ref{26}) holds for almost every $%
x\in \lbrack 0,1)$.

The equality (\ref{27}) follows directly from (\ref{25}) and (\ref{28}).

Let $A$ be a set of all points $x\in \lbrack 0,1)$ for which (\ref{26}) does
not hold. We have that $|A|=0$. Let us consider polynomials $Q_{r}$ $(r\in 
\mathbb{N})$ with the properties: 
\begin{equation}
\sum_{r=1}^{\infty }\Vert Q_{r}\Vert _{L^{2}}<\infty ;  \label{29}
\end{equation}%
\begin{equation}
\sup_{r\in \mathbb{N}}|Q_{r}(x)|=\infty \;\;\text{for every}\;\;x\in A.
\label{30}
\end{equation}%
Such polynomials are constructed for example in \cite{24} (see the proof of
Theorem 3).

For every $r\in\mathbb{N}$ let us find a number $j(r)\in\mathbb{N}$ for
which $\text{deg}(Q_r)\leq n_{\alpha(j(r))}$ and denote $\delta(r)=n_{%
\beta(j(r))}-n_{\alpha(j(r))}$ and $Q^\ast_r=w_{\delta(r)}Q_r$.

Since, $|Q_{r}^{\ast }|=|Q_{r}|$ then by (\ref{29}), $\sum_{r=1}^{\infty
}\Vert Q_{r}^{\ast }\Vert _{L^{2}}<\infty $. Hence, we can introduce the
function $f_{\ast }$ which is the sum of the series $\sum_{r=1}^{\infty
}Q_{r}^{\ast }$ in the space $L^{2}[0,1)$.

By the condition $\text{Sp}(n_{\beta (j(r))})\cap \lbrack 0,\max \text{Sp}%
(n_{\alpha (j(r))})]=\text{Sp}(n_{\alpha (j(r))})$ we conclude that $%
n_{\alpha (j(r))}<\delta (r)$ and that each number $h\leq n_{\alpha (j(r))}$
has the spectrum disjoint with the spectrum of $\delta (r)$. Consequently,
we have 
\begin{equation}
\text{Sp}(Q_{r}^{\ast })\subset (n_{\alpha (j(r))},n_{\beta (j(r))}];
\label{31}
\end{equation}%
\begin{equation}
S_{n_{\beta (j(r))}}(Q_{r}^{\ast })=w_{\delta (r)}Q_{r};  \label{32}
\end{equation}%
\begin{equation}
S_{n_{\alpha (j(r))}}(Q_{r}^{\ast })=0.  \label{33}
\end{equation}

For every $r\in \mathbb{N}$ and $x\in \lbrack 0,1)$ by (\ref{31})-(\ref{33})
we obtain 
\begin{equation*}
S_{n_{\beta (j(r))}}(f_{\ast })(x)-S_{n_{\alpha (j(r))}}(f_{\ast
})(x)=S_{n_{\beta (j(r))}}(Q_{r}^{\ast })(x)-S_{n_{\alpha
(j(r))}}(Q_{r}^{\ast })(x)=
\end{equation*}%
\begin{equation*}
=w_{\delta (r)}(x)Q_{r}\left( x\right) .
\end{equation*}%
Consequently, by virtue of (\ref{30}) we conclude 
\begin{equation}
\sup_{r\in \mathbb{N}}|S_{n_{\beta (j(r))}}(f_{\ast })(x)-S_{n_{\alpha
(j(r))}}(f_{\ast })(x)|=\infty  \label{34}
\end{equation}%
for every $x\in A$.

Note that for every $j\in \mathbb{N}$, $m,n\in \lbrack n_{\beta (j)},N_{\nu
(j+1)}]$ and $x\in \lbrack 0,1)$ by (\ref{31}) and (\ref{25}) we have: 
\begin{equation}
S_{n}(f_{\ast })(x)-S_{m}(f_{\ast })(x)=0.  \label{35}
\end{equation}

Set $f=f^{\ast }+f_{\ast }$. Obviously, $f\in \varphi (L)[0,1)$. By virtue
of (\ref{26}), (\ref{27}), (\ref{34}) and (\ref{35}) we easily see that 
\begin{equation*}
\sup_{k\in \mathbb{N}}|S_{n_{k}}(f)(x)|=\infty
\end{equation*}%
for every point $x\in \lbrack 0,1)$. The theorem is proved.

\end{document}